\documentclass[11pt]{article}

\RequirePackage[OT1]{fontenc}
\RequirePackage{amsthm,amsmath,amssymb,amsfonts,graphicx}
\RequirePackage[colorlinks]{hyperref}
\RequirePackage{hypernat}

\numberwithin{equation}{section}
\theoremstyle{plain}
\newtheorem{thm}{Theorem}[section]

\newtheorem{prop}[thm]{Proposition}
\newtheorem{cor}[thm]{Corollary}
\newtheorem{lem}[thm]{Lemma}
\newtheorem{defn}[thm]{Definition}

\newcommand{\blah}[1]{}

\newcommand{\bea}{\begin{eqnarray}}
\newcommand{\eea}{\end{eqnarray}}
\newcommand{\bean}{\begin{eqnarray*}}
\newcommand{\eean}{\end{eqnarray*}}
\def\be{\begin{eqnarray}}
\def\ee{\end{eqnarray}}
\def\ben{\begin{eqnarray*}}
\def\een{\end{eqnarray*}}

\newcommand{\Sumn}{\sum_{i=1}^{n}}
\newcommand{\ra} {\rightarrow}

\newcommand{\suminfj}{\sum_{j=1}^{\infty}}

\newcommand{\Suminfs}{\sum_{s=0}^{\infty}}

\newcommand{\Zpl}{\mathbb{Z}_{+}}
\newcommand{\NN}{{\mathbb N}}
\newcommand{\IND}{{\mathbb I}}
\newcommand{\RL}{{\mathbb R}}

\newcommand{\Bern}{\mbox{\rm Bern}}

\newcommand{\stv}{{\mbox{\rm\scriptsize TV}}}

\newcommand{\dou}{\partial}
\newcommand{\half}{{\textstyle \frac{1}{2}}}

\newcommand{\lam}{\lambda}

\newcommand{\vc}[1]{{\mathbf #1}}

\newcommand{\Cpo}{\mbox{\rm CPo}(\lambda,Q)}
\newcommand{\Pol}{\mbox{\rm Po}(\lam)}

\def\l{\lambda}
\def\Bl{\left(}
\def\Br{\right)}
\def\th{\theta}
\def\sth{{\sqrt{\th}}}
\def\n{\nu}
\def\a{\alpha}

\begin{document}

\title{Compound Poisson Approximation\\
via Information Functionals\protect\thanks{Portions 
of this paper are based on the Ph.D.\
dissertation \cite{Mad05:phd} of M.\ Madiman, advised by I.\ Kontoyiannis,
at the Division of Applied Mathematics, Brown University.}
}

\author
{
	A.D. Barbour
	\thanks{
	Angewandte Mathematik,
	Universit\"at Z\"urich--Irchel,
	Winterthurerstrasse 190,
	CH--8057 Z\"urich, Switzerland.
	Email: \texttt{a.d.barbour@math.unizh.ch}
	}
\and
	O. Johnson
	\thanks{
	Department of Mathematics,
	University of Bristol,
	University Walk, Bristol, BS8 1TW, UK.
	Email: \texttt{
	O.Johnson@bristol.ac.uk
		}
	}
\and 
	I. Kontoyiannis
	\thanks{
	Department of Informatics,
	Athens University of Economics \& Business,
	Patission 76, Athens 10434, Greece.
	Email: \texttt{
	yiannis@aueb.gr
		}.
	I.K.\ was supported in part
	by a Marie Curie International
	Outgoing Fellowship, PIOF-GA-2009-235837.
	}
\and
	M. Madiman
	\thanks{
	Department of Statistics,
	Yale University,
	24 Hillhouse Avenue,
	New Haven CT 06511, USA.
	Email: \texttt{
	mokshay.madiman@yale.edu
		}
	}
}

\maketitle

\begin{abstract}
An information-theoretic development is given
for the problem of compound Poisson approximation,
which parallels earlier treatments for Gaussian
and Poisson approximation.
Let $P_{S_n}$ be the distribution of a sum
$S_n=\Sumn Y_i$ of independent integer-valued
random variables $Y_i$. Nonasymptotic bounds
are derived for the distance between $P_{S_n}$
and an appropriately chosen compound Poisson law.
In the case where all $Y_i$ have the same
conditional distribution given $\{Y_i\neq 0\}$,
a bound on the relative entropy distance between
$P_{S_n}$ and the compound Poisson distribution
is derived, based on the data-processing property
of relative entropy and earlier Poisson approximation
results. When the $Y_i$ have arbitrary distributions,
corresponding bounds are derived in terms of the total
variation distance. The main technical ingredient 
is the introduction of two ``information 
functionals,'' and the analysis of their properties.
These information functionals play a role analogous 
to that of the classical Fisher information 
in normal approximation.
Detailed comparisons are made
between the resulting inequalities and 
related bounds.
\end{abstract}


\bigskip






\section{Introduction and main results}
\label{sec:intro}
The study of the distribution of a sum
$S_n=\Sumn Y_i$ 
of weakly dependent random
variables $Y_i$ is an important part of
probability theory, with numerous classical
and modern applications. 
This work provides an information-theoretic
treatment of the problem of approximating the 
distribution of $S_n$ by a compound Poisson law,
when the $Y_i$ are discrete, independent random 
variables. Before describing the present approach,
some of the relevant
background is briefly reviewed.

\subsection{Normal approximation and entropy}
\label{sec:normal}
When $Y_1,Y_2,\ldots,Y_n$ are independent
and identically distributed (i.i.d.)
random variables with mean zero
and variance $\sigma^2<\infty$, the central limit
theorem (CLT) and its various refinements state
that the distribution of $T_n:=(1/\sqrt{n})\Sumn Y_i$ 
is close to the $N(0,\sigma^2)$ distribution for 
large $n$.
In recent years the CLT
has been examined from an information-theoretic
point of view and, among various results, 
it has been shown that, if the $Y_i$ have a 
density with respect to Lebesgue measure, 
then the density $f_{T_n}$ of the normalized
sum $T_n$ converges {\em monotonically} to the 
normal density with mean zero and variance $\sigma^2$;
that is, the entropy 
$h(f_{T_n}):=-\int f_{T_n}\log f_{T_n}$ 
of $f_{T_n}$ {\em increases} to the $N(0,\sigma^2)$
entropy as $n\to\infty$, 
which is {\em maximal} among all
random variables with fixed variance $\sigma^2$.
[Throughout, `$\log$' denotes the natural logarithm.]

Apart from this intuitively appealing result,
information-theoretic ideas and techniques have
also provided nonasymptotic inequalities,
for example giving accurate bounds on the relative
entropy $D(f_{T_n}\|\phi):=\int f_{T_n}\log(f_{T_n}/\phi)$
between the density of $T_n$ and the limiting 
normal density $\phi$. Details can be found in
\cite{Bar86,JB04,Joh04:book,ABBN04:1,ABBN04:2,TV06,MB07}
and the references in these works. 

The gist of the information-theoretic approach is
based on estimates of the Fisher information, 
which acts as a ``local'' version of the relative entropy. 
For a random variable $Y$ with a differentiable 
density $f$ and variance $\sigma^2<\infty$, 
the {\em (standardized) Fisher information} is defined as,
\ben
J_{N}(Y):= E \bigg[  \frac{\dou}{\dou y} 
\log f(Y)- \frac{\dou}{\dou y} \log \phi(Y) \bigg]^{2} ,
\een
where $\phi$ is the $N(0,\sigma^2)$ density.
The functional $J_N$ satisfies the
following properties:
\begin{enumerate}
\item[(A)] $J_N(Y)$ is the variance of 
the (standardized) score function,
$r_Y(y):= \frac{\dou}{\dou y} \log f(y)- \frac{\dou}{\dou y} \log \phi(y)$,
$y\in\RL$.
\item[(B)] $J_{N}(Y)=0$ if and only if $Y$ is Gaussian.
\item[(C)] $J_{N}$ satisfies a subadditivity property for sums.
\item[(D)] If $J_{N}(Y)$ is small then the density $f$ 
of $Y$ is approximately normal and, in particular,
$D(f\| \phi)$ is also appropriately small.
\end{enumerate}
Roughly speaking, the information-theoretic approach
to the CLT and associated normal approximation bounds 
consists of two steps;
first a strong version of Property~(C) is used
to show that $J_{N}(T_n)$ is close to zero for 
large~$n$, and then Property~(D) is applied to obtain 
precise bounds on the relative 
entropy $D(f_{T_n}\|\phi)$.

\subsection{Poisson approximation}
More recently, an analogous program was 
carried out for Poisson approximation.
The Poisson law was identified
as having maximum entropy within a natural
class of discrete distributions on~$\Zpl:=\{0,1,2,\ldots\}$
\cite{Har01,Top02,Joh07}, and Poisson
approximation bounds in terms of relative
entropy were developed in \cite{KHJ05};
see also \cite{JM87} for earlier related results.
The approach of \cite{KHJ05} follows a similar
outline to the one described above for
normal approximation. Specifically, 
for a random variable
$Y$ with values in $\Zpl$
and distribution $P$, the 
{\em scaled Fisher information of Y}
was defined as,
\be
J_{\pi}(Y) :=\lam E[\rho_Y(Y)^2]=\lambda{\rm Var}(\rho_Y(Y)),
\label{eq:scaledF}
\ee
where $\lambda$ is the mean of $Y$ and
the scaled score function $\rho_Y$ is given by,
\be
\rho_Y(y) 
:= \frac{ (y + 1) P(y + 1) }{\lam P(y)} -1,
\;\;\;\;y\geq 0.
\label{eq:scaledS}
\ee
[Throughout, we use the term `distribution'
to refer to the discrete probability mass
function of an integer-valued random variable.]

As discussed briefly before the proof 
of Theorem~\ref{thm:1} in Section~\ref{ss:simple}
the functional $J_{\pi}(Y)$ was shown in
\cite{KHJ05} to satisfy Properties~(A-D) exactly 
analogous to those of the Fisher
information described above, with the Poisson law
playing the role of the Gaussian distribution.
These properties were employed 
to establish optimal or near-optimal Poisson 
approximation bounds for the distribution of 
sums of nonnegative integer-valued random variables
\cite{KHJ05}.

\subsection{Compound Poisson approximation}
This work provides a parallel treatment 
for the more general -- and technically significantly
more difficult -- problem of approximating
the distribution $P_{S_n}$ of a sum
$S_n=\Sumn Y_i$ of independent
$\Zpl$-valued random variables by an appropriate
compound Poisson law. 
This and related questions
arise naturally in applications involving 
counting; see, e.g., 
\cite{BHJ92:book,Ald89:book,barbour:monograph,diaconis:stein}.
As we will see, in this setting 
the information-theoretic approach
not only gives an elegant alternative 
route to the classical asymptotic results
(as was the case in the first information-theoretic 
treatments of the CLT), 
but it actually yields fairly sharp
finite-$n$ inequalities that are competitive
with some of the best existing bounds.

Given a distribution $Q$ on 
$\NN=\{1,2,\ldots\}$ and a $\lambda>0$,
recall that the compound Poisson law
$\Cpo$ is defined as the distribution
of the random sum $\sum_{i=1}^Z X_i$,
where $Z\sim\Pol$ is Poisson distributed
with parameter $\lambda$ and the  
$X_i$ are i.i.d.\ with distribution $Q$,
independent of $Z$.

Relevant results that can be seen as the
intellectual background to the 
information-theoretic approach for
compound Poisson approximation were 
recently established 
in \cite{JKM09:pre,YYu:pre},
where it was shown that, like the Gaussian 
and the Poisson, the compound Poisson law 
has a maximum entropy property within a natural 
class of probability measures on $\Zpl$. 
Here we provide nonasymptotic,
computable and accurate bounds for the
distance between $P_{S_n}$ and an appropriately
chosen compound Poisson law, partly based on
extensions of the 
information-theoretic techniques 
introduced in \cite{KHJ05} and \cite{JM87}
for Poisson approximation.

In order to state our main results we
need to introduce some more terminology.
When considering the distribution of $S_n = \Sumn Y_i$,
we find it convenient 
to write each $Y_i$ as the product $B_i X_i$ of two
independent random variables, where $B_{i}$ takes 
values in $\{0,1\}$ and $X_i$ takes values
in $\NN$. This is done uniquely 
and without loss of generality, by taking 
$B_i$ to be $\Bern(p_i)$ with 
$p_i=\Pr\{Y_i\neq 0\}$,
and $X_i$ having distribution 
$Q_i$ on $\NN$,
where $Q_i(k)=\Pr\{Y_i=k\,|\,Y_i\geq 1\}=\Pr\{Y_i=k\}/p_i$,
for $k\geq 1$. 

In the special case of a sum 
$S_n=\Sumn Y_i$ of random variables $Y_i=B_iX_i$ 
where all the $X_i$ have the same distribution $Q$, 
it turns out that the problem of approximating 
$P_{S_n}$ by a compound Poisson law can
be reduced to a Poisson approximation
inequality. This is achieved by an application
of the so-called ``data-processing'' property of
the relative entropy, which then facilitates the use
of a Poisson approximation bound established
in \cite{KHJ05}. The result is stated in 
Theorem~\ref{thm:1} below;
its proof is 
given in Section~\ref{ss:simple}.

\begin{thm}
\label{thm:1}
Consider a sum $S_n = \Sumn Y_i$ 
of independent random variables $Y_i=B_iX_i$,
where the $X_i$ are i.i.d.\ $\sim Q$
and the $B_i$ are independent $\Bern(p_i)$. Then 
the relative entropy between the distribution $P_{S_n}$
of $S_n$ and the $\Cpo$ distribution 
satisfies,
\ben
D(P_{S_n} \| \Cpo) \leq \frac{1}{\lam} \sum_{i=1}^n \frac{p_i^3}{1-p_i},
\een
where $\lam := \sum_{i=1}^n p_i$.
\end{thm}

Recall that, for distributions $P$ and $Q$ on $\Zpl$, 
the {\em relative entropy}, or {\em Kullback-Leibler 
divergence}, $D(P\|Q)$,
is defined by,
\ben
D(P\|Q):=\sum_{x\in \Zpl}P(x)\log\Big[\frac{P(x)}{Q(x)}\Big].
\een
Although not a metric, relative entropy
is an important measure of closeness between
probability distributions \cite{CT91:book}\cite{CK81:book}
and it can be used to obtain total variation bounds 
via Pinsker's inequality \cite{CK81:book},
\ben
d_\stv(P,Q)^2\leq \half D(P\|Q),
\een
where, as usual,
the total variation distance is
\ben 
d_\stv(P,Q) := \half\sum_{x\in \Zpl} \big|P(x)-Q(x)\big| 
= \max_{A\subset \Zpl} \big|P(A)-Q(A)\big| .
\een

In the general case where the distributions
$Q_i$ corresponding to the $X_i$ in the 
summands $Y_i=B_i X_i$ are not identical,
the data-processing argument used in the proof 
of Theorem~\ref{thm:1} can no longer 
be applied. Instead, the key idea 
in this work is the introduction of two 
``information functionals,'' or simply 
``informations,'' which, in the present context,
play a role analogous
to that of the Fisher information $J_N$ 
and the scaled 
Fisher information $J_{\pi}$ in Gaussian 
and Poisson approximation, respectively.

In Section~\ref{ss:lidefn}
we will define two such 
information functionals, 
$J_{\vc{Q},1}$ and $J_{Q,2}$,
and use them to derive compound Poisson
approximation bounds. Both 
$J_{\vc{Q},1}$ and $J_{Q,2}$
will be seen to satisfy natural analogs 
of Properties~(A-D) stated above, except 
that only a weaker version 
of Property~(D) will be established: When
either 
$J_{\vc{Q},1}(Y)$ or $J_{Q,2}(Y)$
is close to zero, the distribution 
of $Y$ is close to a compound Poisson law 
in the sense of total variation rather than
relative entropy. As in normal and Poisson 
approximation, combining the analogs
of Properties~(C) and~(D) 
satisfied by the two new information functionals,
yields new compound Poisson 
approximation bounds.

\begin{thm}
\label{thm:2}
Consider a sum $S_n = \Sumn Y_i$ 
of independent random variables $Y_i=B_iX_i$,
where each $X_i$ has distribution $Q_i$ on $\NN$
with mean $q_i$, and each $B_i\sim\Bern(p_i)$.
Let $\lam = \sum_{i=1}^n p_i$
and $Q=\sum_{i=1}^n\frac{p_i}{\lambda}Q_i$.
Then,
\ben
d_\stv(P_{S_n}, \Cpo )\leq 
H(\lambda,Q) q \left\{ \left[ \sum_{i=1}^n \frac{p_i^3}{1-p_i} 
\right]^{1/2} + D( \vc{Q}) \right\},
\een
where $P_{S_n}$ is the distribution of $S_n$,
$q=\sum_{i=1}^n\frac{p_i}{\lambda}q_i$,
$H(\lambda,Q)$ denotes the Stein factor defined
in {\em (\ref{eq:steinF})} below,
and $D(\vc{Q})$ is a 
measure of the dissimilarity of the 
distributions $\vc{Q}=(Q_i)$, which 
vanishes when the $Q_i$ are identical:
\be
D(\vc{Q}) := \sum_{j=1}^\infty \sum_{i=1}^n \frac{jp_i }{q}
\, |Q_i(j)-Q(j)|\,.
\label{eq:DQ}
\ee
\end{thm}

Theorem~\ref{thm:2} is an immediate
consequence of the subadditivity 
property of $J_{\vc{Q},1}$ 
established in Corollary~\ref{cor:j1simple},
combined with the total variation bound
in Proposition~\ref{prop:sb-stein}.
The latter bound states that,
when $J_{\vc{Q},1}(Y)$ is small,
the total variation distance between
the distribution of $Y$ and a compound
Poisson law is also appropriately small. 
As explained in Section~\ref{ss:tvbd},
the proof of Proposition~\ref{prop:sb-stein}
uses a basic result that comes up
in the proof of compound Poisson inequalities
via Stein's method, namely, a bound on the
sup-norm of the solution of the Stein equation.
This explains the appearance of the 
Stein factor, defined next. 
But we emphasize that, apart from this 
point of contact, the overall methodology
used in establishing the results in Theorems~\ref{thm:2}
and~\ref{thm:3} is entirely different from that
used in proving compound Poisson approximation
bounds via Stein's method.

\begin{defn}
\label{defn:SteinF}
Let $Q$ be a distribution on $\NN$. 
If $\{jQ(j)\}$ is a non-increasing sequence, set 
$\delta= [\lam \{ Q(1)-2Q(2) \}]^{-1}$ and let,
\ben
H_0(\lam, Q) = \left\{ \begin{array}{ll} 
1 & \textrm{if $\delta\geq 1$}\\ 
\sqrt{\delta}(2-\sqrt{\delta}) & \textrm{if $\delta<1$.}
\end{array} \right. 
\een
For general $Q$ and any $\lam>0$, the 
{\em Stein factor $H(\lambda,Q)$} is defined as:
\be
\hspace{0.3in}
H(\lam,Q)= \left\{ \begin{array}{ll} 
H_0(\lam, Q), & \textrm{if $\{jQ(j)\}$ is non-increasing}\\ 
e^{\lam} \min \bigg\{ 1, \frac{1}{\lam Q(1)} \bigg\}, & \textrm{otherwise.}
\end{array} \right. 
\label{eq:steinF}
\ee
\end{defn}

Note that in the case when all the $Q_i$ are identical,
Theorem~\ref{thm:2} yields,
\be
d_\stv(P_{S_n}, \Cpo)^2\leq 
H(\lambda,Q)^2\,q^2\,
\sum_{i=1}^n \frac{p_i^3}{1-p_i},
\label{eq:ident1}
\ee
where $q$ is the common mean of the $Q_i=Q$,
whereas Theorem~\ref{thm:1} combined with 
Pinsker's inequality yields a similar,
though not generally comparable, bound,
\be
d_\stv(P_{S_n},\Cpo)^2
\leq 
\frac{1}{2\lambda}
\sum_{i=1}^n \frac{p_i^3}{1-p_i}.
\label{eq:ident2}
\ee
See Section~\ref{sec:num} for detailed
comparisons in special cases.

The third and last main result,
Theorem~\ref{thm:3}, gives an analogous
bound to that of Theorem~\ref{thm:2},
with only a single term in the right-hand-side.
It is obtained from the subadditivity 
property of the second information
functional $J_{\vc{Q},2}$,
Proposition~\ref{prop:proj2},
combined with the corresponding 
total variation bound
in Proposition~\ref{prop:jloc}.

\begin{thm}
\label{thm:3}
Consider a sum $S_n = \Sumn Y_i$ 
of independent random variables $Y_i=B_iX_i$,
where each $X_i$ has distribution $Q_i$ 
on $\NN$ with mean
$q_i$, and each $B_i\sim\Bern(p_i)$.
Assume all $Q_i$ have have full support on $\NN$, 
and let $\lam = \sum_{i=1}^n p_i$,
$Q=\sum_{i=1}^n\frac{p_i}{\lambda}Q_i$,
and $P_{S_n}$
denote the distribution of $S_n$. Then,
$$
d_\stv(P_{S_n} ,\Cpo)\leq 
H(\lambda,Q)
\left\{
\sum_{i=1}^n 
\left[
p_i^3 
\sum_y Q_i(y) y^2
\Big( \frac{ Q_i^{*2}(y)}{2 Q_i(y)} - 1 \Big)^2
\right]\right\}^{1/2},
$$
where $Q_i^{*2}$ denotes the convolution $Q_i*Q_i$
and $H(\lambda,Q)$ denotes the Stein factor
defined in {\em (\ref{eq:steinF})} above.
\end{thm}

The accuracy of the bounds in the three
theorems above is examined in specific 
examples in Section \ref{sec:num}, where
the resulting estimates are compared with 
what are probably the sharpest known bounds 
for compound Poisson approximation.
Although the main conclusion of these comparisons --
namely, that in broad terms our bounds are
competitive with some of the best existing bounds and,
in certain cases, may even be the sharpest --
is certainly encouraging, we wish to emphasize
that the main objective of this work is the development
of an elegant conceptual framework for compound 
Poisson limit theorems via information-theoretic
ideas, akin to the remarkable information-theoretic 
framework that has emerged for the central limit theorem
and Poisson approximation. 

\newpage

The rest of the paper is organized as follows. 
Section~\ref{ss:simple} contains basic facts,
definitions and notation that will remain in 
effect throughout. It also contains a brief review 
of earlier Poisson approximation results in terms
of relative entropy, and the proof of Theorem~\ref{thm:1}.
Section~\ref{ss:lidefn} introduces the two 
new
information functionals:
The {\em size-biased information} $J_{\vc{Q},1}$,
generalizing the scaled Fisher information 
of \cite{KHJ05}, and the {\em Katti-Panjer
information} $J_{Q,2}$, generalizing a related 
functional introduced by Johnstone and 
MacGibbon in \cite{JM87}. 
It is shown that, in each case,
Properties~(A) and~(B) analogous to those
stated in Section~\ref{sec:normal} for Fisher 
information hold for 
$J_{\vc{Q},1}$ and $J_{Q,2}$.
In Section~\ref{ss:subadd} we consider Property~(C)
and show that both $J_{\vc{Q},1}$ and $J_{Q,2}$ satisfy
natural subadditivity properties on convolution.
Section~\ref{ss:tvbd} contains bounds analogous
to that Property~(D) above, showing that
both $J_{\vc{Q},1}(Y)$ and $J_{Q,2}(Y)$ 
dominate the total variation distance between
the distribution of $Y$ and a compound Poisson law.


\section{Size-biasing, compounding and relative entropy}
\label{ss:simple}

In this section we collect preliminary definitions
and notation that will be used in subsequent sections,
and we provide the proof of Theorem~\ref{thm:1}.

The compounding operation in the definition of the 
compound Poisson law in the Introduction can be more
generally phrased as follows.
[Throughout, the empty sum $\sum_{i=1}^0 [\ldots]$
is taken to be equal to zero]. 
\begin{defn}\label{defn:cpding}
For any $\Zpl$-valued random variable $Y\sim R$
and any distribution $Q$ on $\NN$,
the {\em compound distribution $C_{Q}R$} is
that of the sum,
\ben
\sum_{i=1}^{Y} X_i ,
\een
where the $X_i$
are i.i.d.\
with common distribution $Q$, independent of $Y$.
\end{defn}

For example, the compound Poisson law $\Cpo$ 
is simply $C_Q\Pol$, and the {\em compound binomial
distribution} $C_Q$Bin($n,p$) is that of
the sum $S_n=\Sumn B_iX_i$ where the $B_i$ are i.i.d.\
Bern$(p)$ and the $X_i$ are i.i.d.\ with distribution
$Q$, independent of the $B_i$.
More generally, if the $B_i$ are Bernoulli
with different parameters $p_i$, we say that
$S_n$ is a {\em compound Bernoulli sum}
since the distribution of each summand $B_iX_i$
is $C_{Q}\Bern(p_i)$.

Next we recall the size-biasing operation, which is
intimately related to the Poisson law.
For any distribution $P$ on $\Zpl$ with mean $\lambda$,
the {\em (reduced) size-biased distribution $P^{\#}$} is,
$$P^{\#}(y)=\frac{(y+1)P(y+1)}{\lam},
\;\;\;\;y\geq 0.$$
Recalling that a distribution $P$ on $\Zpl$
satisfies the recursion,
\be
(k+1)P(k+1):=\lambda P(k),
\;\;\;\;k\in \Zpl,
\label{eq:PoissonR}
\ee
if and only if $P=\Pol$, it is immediate that
$P=\Pol$ if and only if $P=P^{\#}$. 
This also explains, in part, the definition
(\ref{eq:scaledF}) of the scaled Fisher information
in \cite{KHJ05}. Similarly, the {\em Katti-Panjer recursion}
states that $P$ is the $\Cpo$ law if and only if,
\be
\label{eq:panjer}
k P(k) = \lam \sum_{j=1}^{k} j Q(j) P (k-j),
\;\;\;\; k\in \Zpl;
\ee
see the discussion in \cite{JKM09:pre} 
for historical remarks on the origin of (\ref{eq:panjer}).

Before giving the proof of Theorem~\ref{thm:1} we
recall two results related to Poisson approximation bounds
from \cite{KHJ05}. First, for any random
variable $X\sim P$ on $\Zpl$ with mean $\lambda$,
a modified log-Sobolev inequality of \cite{BL98}
was used in \cite[Proposition~2]{KHJ05}
to show that,
\be
D(P\|\Pol)\leq J_\pi(X),
\label{eq:logS}
\ee
as long as $P$ has either full support
or finite support. Combining this with the
subadditivity property of $J_\pi$ and
elementary computations, yields 
\cite[Theorem~1]{KHJ05} 
that states:
If $T_n$ is the sum of $n$ independent 
$B_i\sim\Bern(p_i)$ random variables, 
then,
\be
D(P_{T_n}\|\Pol)\leq\frac{1}{\lambda}\sum_{i=1}^n\frac{p_i^3}{1-p_i},
\label{poi-bound}
\ee
where $P_{T_n}$ denotes the distribution of $T_n$
and $\lambda=\sum_{i=1}^np_i$.

\begin{proof}[Proof of Theorem~\ref{thm:1}]
Let $Z_n \sim\Pol$ and $T_{n}=\sum_{i=1}^{n} B_{i}$.
Then the distribution of $S_n$ is also that of the sum
$\sum_{i=1}^{T_n} X_i$; similarly,
the $\Cpo$ law is the distribution of
the sum $Z=\sum_{i=1}^{Z_n} X_i$.
Thus, writing $\vc{X}=(X_i)$, we can express
 $S_n = f(\vc{X},T_n)$ and 
 $Z = f(\vc{X},Z_n)$,
where the function $f$ is the same in both places.
Applying the data-processing inequality 
and then the chain rule for relative entropy
\cite{CK81:book},
\ben
D(P_{S_n}\|\Cpo)
&\leq & D(P_{\vc{X},T_n}\|P_{\vc{X},Z_n} ) \\
&= & \Big[\sum_{i} D(P_{X_i}\|P_{X_i})\Big] + D(P_{T_n}\|P_{Z_n}) \\
&=& D(P_{T_n}\|\Pol),
\een
and the result follows from the Poisson approximation bound
\eqref{poi-bound}.
\end{proof}

\newpage

\section{Information functionals} 
\label{ss:lidefn}

This section contains the definitions 
of two new information functionals
for discrete random variables, along 
with some of their basic properties.

\subsection{Size-biased information}
\label{ss:SBLI-defn}

For the first information functional we consider, 
some knowledge of the summation structure 
of the random variables concerned is required.

\begin{defn}\label{defn:SBLI-basic}
Consider the sum
$S = \sum_{i=1}^n Y_i \sim P$
of $n$ independent $\Zpl$-valued random variables
$Y_i\sim P_i= C_{Q_i} R_i$, $i=1,2,\ldots,n$.
For each $j$, let $Y_j'\sim C_{Q_j}(R_j^\#)$ 
be independent of the $Y_i$, and let
$S^{(j)}\sim P^{(j)}$ be the same sum as $S$
but with $Y_j'$ in place of $Y_j$.

Let $q_i$ denote the mean of each $Q_i$,
$p_i = E(Y_i)/q_i$ and 
$\lam = \sum_i p_i$. Then 
the {\em size-biased information of $S$ relative to
the sequence $\vc{Q}=(Q_i)$} is,
\ben
J_{\vc{Q},1}(S):=\lambda E[r_1(S;P,\vc{Q})^2],
\een
where the score function $r_1$ is defined by,
\ben
r_{1}(s; P, \vc{Q}):= \frac{ \sum_{i} p_i P^{(i)}(s)}
{\lam P(s)} - 1,
\;\;\;\;s\in\Zpl.
\een
\end{defn}

For simplicity,
in the case of a single summand $S=Y_1\sim P_1=C_QR$
we write $r_1(\cdot;P,Q)$ and
$J_{Q,1}(Y)$ for the score and the size-biased
information of $S$, respectively.
[Note that the score function $r_1$ is only infinite
at points $x$ outside the support of $P$, which do not
affect the definition of the size-biased information
functional.]

Although at first sight the definition of 
$J_{\vc{Q},1}$ seems restricted to the case
when all the summands $Y_i$ have 
distributions of the form $C_{Q_i}R_i$,
we note that this can always be achieved
by taking $p_i=\Pr\{Y_i\geq 1\}$ and letting
$R_i\sim\Bern(p_i)$ and $Q_i(k)=\Pr\{Y_i=k|Y_i\geq 1\}$,
for $k\geq 1$, as before.

We collect below some of the basic properties 
of $J_{\vc{Q},1}$ that follow easily from
the definition.
\begin{enumerate}
\item{Since $E[r_1(S; P, \vc{Q})]=0$, the
functional $J_{\vc{Q},1}(S)$ is in fact the 
variance of the score $r_1(S; P, \vc{Q})$.}
\item{In the case of a single summand $S=Y_1\sim C_QR$,
if $Q$ is the point mass at 1
then the score $r_1$ reduces to the 
score function $\rho_Y$ in (\ref{eq:scaledS}).
Thus $J_{\vc{Q},1}$ can be seen as
a generalization of the scaled Fisher information $J_{\pi}$
of \cite{KHJ05}
defined in (\ref{eq:scaledF}).}
\item{Again in 
the case of a single summand $S=Y_1\sim C_Q R$,
we have that $r_1(s;P,Q) \equiv 0$ if and only
if $R^{\#}=R$, i.e., if and only if $R$ is the 
$\Pol$ distribution.
Thus in this case $J_{Q,1}(S)=0$  
if and only if $S\sim\Cpo$ for some $\lam>0$.}
\item{In general, writing $F^{(i)}$ for the distribution
of the leave-one-out sum $\sum_{j\neq i}Y_i$,
$$ r_1( \cdot; P,\vc{Q}) \equiv 0 \iff \sum p_i F^{(i)} * (C_{Q_i} R_i  - 
C_{Q_i} R_i^{\#}) \equiv  0.$$
Hence within the class of ultra log-concave $R_i$ (a class
which includes compound Bernoulli sums), since the moments 
of $R_i$ are no smaller than the moments of
$R_i^{\#}$ with equality if and only if
$R_i$ is Poisson,
the score $r_1(\cdot;P,\vc{Q}) \equiv 0$ if and only
if the $R_i$ 
are all Poisson, i.e., if and only if $P$ is 
compound Poisson.}
\end{enumerate}

\subsection{Katti-Panjer information}
\label{ss:JMLI-defn}

Recall that the recursion (\ref{eq:PoissonR})
characterizing the Poisson distribution
was used as part of the motivation for the
definition of the scaled Fisher information 
$J_\pi$ in (\ref{eq:scaledF}) and (\ref{eq:scaledS}).
In an analogous manner, we employ the 
Katti-Panjer recursion (\ref{eq:panjer})
that characterizes the compound Poisson
law to define another information
functional.

\begin{defn} 
Given a $\Zpl$-valued random variable $Y\sim P$ 
and an arbitrary distribution $Q$ on $\NN$,
the {\em Katti-Panjer information of $Y$ 
relative to $Q$} is defined as,
\ben 
J_{Q,2}(Y) := E[ r_{2}(Y;P,Q)^2],
\een
where the score function $r_{2}$ is,
\ben 
r_{2}(y;P,Q) 
:= \frac{\lam \sum_{j=1}^{\infty} j Q(j) P(y-j)}{P(y)} - y ,
\;\;\;\;y\in\Zpl,
\een
and where $\lambda$ is the ratio of the mean of $Y$ to 
the mean of $Q$.
\end{defn}

From the definition of the score function
$r_2$ it is immediate that,
\ben
E[ r_{2}(Y;P,Q)] &= &\sum_y P(y) r_{2}(y;P,Q) \\ 
&= &\lam
\bigg[ \sum_{y: P(y) > 0} \sum_{j} j Q(j) P(y-j) \bigg] - E(Y)  \\
&= & \lam \bigg[ \sum_{j} j Q(j) \bigg] - E(Y)  = 0,
\een
therefore $J_{Q,2}(Y)$ is equal to the variance
of $r_2(Y;P,Q)$. [This computation 
assumes that $P$ has full support on $\Zpl$;
see the last paragraph of this section for 
further discussion of this point.]
Also, in view of the Katti-Panjer 
recursion (\ref{eq:panjer}) we have
that $J_{Q,2}(Y)=0$ if and only if
$r_{2}(y;P,Q)$ vanishes for all $y$,
which happens if and only if the 
distribution $P$ of $Y$ is $\Cpo$. 

In the special case when $Q$ is the unit
mass at 1, the Katti-Panjer information
of $Y\sim P$ reduces to, 
\be
J_{Q,2}(Y)=E\Big[\Big(\frac{\lambda P(Y-1)}{P(Y)}-Y\Big)^2\Big]
=\lambda^2I(Y)+(\sigma^2-2\lambda),
\label{eq:generalize}
\ee
where $\lambda,\sigma^2$ are the mean and variance of $Y$,
respectively, and
$I(Y)$ denotes the functional,
\be
I(Y):=E\Big[\Big(\frac{P(Y-1)}{P(Y)}-1\Big)^2\Big] ,
\label{eq:JM}
\ee
proposed by Johnstone and MacGibbon \cite{JM87} 
as a discrete version of the Fisher information
(with the convention $P(-1)=0$).
Therefore, in view of (\ref{eq:generalize}) 
we can think of $J_{Q,2}(Y)$ as a generalization 
of the ``Fisher information'' functional $I(Y)$ of \cite{JM87}.

Finally note that, although the definition of 
$J_{Q,2}$ is more straightforward than that of
$J_{\vc{Q},1}$, the Katti-Panjer information 
suffers the drawback that -- like its simpler version
$I(Y)$ in \cite{JM87} --
it is only finite
for random variables $Y$ with full support on $\Zpl$.
As noted in \cite{Kag01} and \cite{KHJ05}, the 
definition of $I(Y)$ cannot simply be extended 
to all $\Zpl$-valued random variables by just
ignoring the points outside the support of $P$,
where the integrand in (\ref{eq:JM}) becomes infinite.
This was, partly, the motivation for the definition
of the scaled scored function $J_\pi$ in \cite{KHJ05}.
Similarly, in the present setting, the important 
properties of $J_{Q,2}$ established in the following
sections {\em fail} unless $P$ has full
support, unlike for the size-biased information 
$J_{\vc{Q},1}$.


\section{Subadditivity}
\label{ss:subadd}

The subadditivity property of Fisher information
(Property~(C) in the Introduction) plays a key role
in the information-theoretic analysis of normal
approximation bounds. The corresponding property
for the scaled Fisher information (Proposition~3 of \cite{KHJ05})
plays an analogous role in the case of Poisson approximation.
Both of these results are based on a convolution identity 
for each of the two underlying score functions.
In this section we develop natural analogs of the
convolution identities and resulting subadditivity
properties for the functionals $J_{\vc{Q},1}$ and
$J_{Q,2}$.

\subsection{Subadditivity of the size-biased information} 
\label{ss:sbLI} 

The proposition below gives the natural analog
of Property~(C) in the the Introduction,
for the information functional $J_{\vc{Q},1}$.
It generalizes the convolution lemma and 
Proposition~3 of \cite{KHJ05}.

\begin{prop}\label{prop:subadd}
Consider the sum
$S_n = \sum_{i=1}^n Y_i \sim P$
of $n$ independent $\Zpl$-valued random variables
$Y_i\sim P_i= C_{Q_i} R_i$, $i=1,2,\ldots,n$.
For each $i$,
let $q_i$ denote the mean of $Q_i$,
$p_i = E(Y_i)/q_i$ and 
$\lam = \sum_i p_i$. Then,
\be
r_1(s; P, \vc{Q}) = E \left[ 
\left. \sum_{i=1}^n \frac{p_i}{\lam} r_1(Y_i; P_i, Q_i) \right|
S_n = s \right],
\label{eq:proj1}
\ee
and hence,
\be
J_{\vc{Q},1}(S_n) \leq \sum_{i=1}^m \frac{p_i}{\lam}
J_{Q_i,1}(Y_i).
\label{eq:sub1}
\ee
\end{prop}
\begin{proof} 
In the notation of Definition~\ref{defn:SBLI-basic}
and the subsequent discussion,
writing $F^{(i)} = P_1 * \ldots * P_{i-1} * P_{i+1} 
* \ldots * P_m$, so that $P^{(i)} = F^{(i)} *
C_{Q_i} R_i^{\#}$, the right-hand side of the 
projection identity (\ref{eq:proj1}) equals,
\begin{eqnarray*}
\lefteqn{ \sum_{i=1}^n \sum_x
\frac{ P_i(x) F^{(i)}(s-x)}{P(s)} 
\left( \frac{p_i}{\lam} \left( \frac{ C_{Q_i} R_i^{\#}(x)}{P_i(x)} - 1
\right) \right) } \\
& = & \frac{1}{\lam P(s)} \left( \sum_{i=1}^n \sum_x p_i C_{Q_i} R_i^{\#}(x) F^{(i)}(s-x) \right) - 1 \\
& = & \frac{1}{\lam P(s)} \left( \sum_{i=1}^n p_i P^{(i)}(s) \right) - 1 \\
& = & r_1(s; P, \vc{Q}),
\end{eqnarray*}
as required.
The subadditivity result follows using the conditional Jensen
inequality, exactly as in the proof of Proposition~3 
of~\cite{KHJ05}.
\end{proof}

\begin{cor} \label{cor:j1simple}
Under the assumptions of Proposition~{\em \ref{prop:subadd}},
if each $Y_i=B_i X_i$, where $B_i\sim\Bern(p_i)$ and
$X_i \sim Q_i$ where $p_i=\Pr\{Y_i\neq 0\}$ and
$Q_i(k)=\Pr\{Y_i=k|Y_i\geq 1\}$, then,
$$ J_{\vc{Q},1}(S_n) 
\leq \frac{1}{\lam} \sum_{i=1}^n \frac{p_i^3}{1-p_i},$$
where $\lam = \sum_i p_i$.
\end{cor}
\begin{proof} Consider $Y = B X$, where $B\sim R=\Bern(p)$
and $X \sim Q$. 
Since $R^{\#}=\delta_{0}$
then $C_{Q}(R^{\#})=\delta_{0}$. Further, $Y$ takes the value
0 with probability $(1-p)$ and the value $X$ with probability $p$.
Thus,
\ben
r_{1}(x;C_Q R,Q) &= &\frac{C_{Q}(R^{\#})(x)}{C_{Q}R(x)}-1 \\
&= & \frac{ \delta_0(x)}{(1-p) \delta_0(x) + p Q(x)}- 1 \\
&= & \left\{ 
\begin{array}{ll} 
\frac{p}{1-p} & \textrm{ for $x=0$}\\ 
-1 & \textrm{ for $x>0$.} 
\end{array} \right. 
\een
Consequently,
\be \label{eq:j1oneterm}
J_{Q,1}(Y) = \frac{p^2}{1-p},
\ee
and using Proposition~\ref{prop:subadd} yields,
$$ J_{\vc{Q},1}( S_n) \leq 
\sum_{i=1}^n \frac{p_i}
{\lam} J_{Q_i,1}(Y_i) = \frac{1}{\lam} \sum_{i=1}^n \frac{p_i^3}{1-p_i},$$ 
as claimed.
\end{proof}

\subsection{Subadditivity of the Katti-Panjer information} 
\label{ss:JMLI}
When $S_n$ is supported on the whole of $\Zpl$,
the score $r_2$ satisfies a convolution 
identity and the functional $J_{Q,2}$ is subadditive. 
The following Proposition contains the analogs 
of (\ref{eq:proj1}) and (\ref{eq:sub1}) 
in Proposition~\ref{prop:subadd}
for the Katti-Panjer information $J_{Q,2}(Y)$.
These  can also be viewed as generalizations
of the corresponding results for the 
Johnstone-MacGibbon functional $I(Y)$ 
established in \cite{JM87}. 

\begin{prop} \label{prop:proj2} 
Consider a sum $S_n = \Sumn Y_i$ 
of independent random variables $Y_i=B_iX_i$,
where each $X_i$ has distribution $Q_i$ on $\NN$
with mean $q_i$, and each $B_i\sim\Bern(p_i)$.
Let $\lam = \sum_{i=1}^n p_i$
and $Q=\sum_{i=1}^n\frac{p_i}{\lambda}Q_i$.
If each $Y_i$ is supported on the whole of $\Zpl$,
then,
$$r_{2}(s;S_{n},Q) = E\left[ \left. \sum_{i=1}^n r_{2}(Y_{i};P_i,Q_i)
\right| S_{n} = s \right],
$$
and hence,
$$ J_{Q,2}(S_n) \leq \sum_{i=1}^n J_{Q_{i},2}(Y_i).$$
\end{prop}
\begin{proof}
Write $r_{2,i}(\cdot)$ for $r_{2}(\cdot;P_{i},Q_{i})$,
and note that $E(Y_i)=p_iq_i$, for each $i$.
Therefore, $E(S_n) = \sum_i p_i q_i$ which equals
$\lam$ times the mean of $Q$. As before,
let $F^{(i)}$ denote the distribution of the
leave-one-out sum $\sum_{j \neq i} Y_j$,
and decompose the distribution $P_{S_n}$
of $S_n$ as $P_{S_n}(s) = \sum_x P_{i}(x) F^{(i)}(s-x)$. 
We have,
\bean
r_{2}(s;S_{n},Q) & = & 
\frac{\lam \sum_{j=1}^{\infty} j Q(j) P_{S_n}(s-j)}{ 
P_{S_n}(s)} - s  \\
& = & \sum_{i=1}^n 
\frac{p_i \sum_{j=1}^{\infty} j Q_i(j) P_{S_n}(s-j)}{P_{S_n}(s)} - s \\
& = & \sum_{i=1}^n \sum_x 
\frac{p_i \sum_{j=1}^{\infty} j Q_i(j) P_{i}(x-j)
F^{(i)}(s-x)}{P_{S_n}(s)} - s \\ 
& = & \sum_{i=1}^n 
\sum_x \frac{ P_{i}(x) F^{(i)}(s-x)}{P_{S_n}(s)} 
\left[ \frac{p_i  \sum_{j=1}^{\infty} j Q_i(j) P_{i}(x-j)}{P_{i}(x)} \right] 
- s
\\
& = & E \left[ \left. \sum_{i=1}^n r_{2,i}(Y_i) \right| S_n = s \right]
\eean
proving the projection identity. And
using the conditional Jensen inequality,
noting that the cross-terms vanish 
because $E[r_{2}(X;P,Q) = 0]$ 
for any $X\sim P$ with full support
(cf.\ the discussion in Section~\ref{ss:JMLI-defn}),
the subadditivity result follows, as claimed.
\end{proof}

\newpage

\section{Information functionals dominate total variation}
\label{ss:tvbd}

In the case of Poisson approximation, 
the modified log-Sobolev inequality (\ref{eq:logS})
directly relates the relative entropy to the 
scaled Fisher information $J_{\pi}$. 
However, the known (modified) log-Sobolev 
inequalities for compound Poisson
distributions \cite{Wu00,KM06}, only relate the relative
entropy to functionals different
from $J_{\vc{Q},1}$ or $J_{Q,2}$.
Instead of developing subadditivity results
for those other functionals, we build,
in part, on some 
of the ideas from Stein's method 
and prove relationships 
between the total variation distance and the 
information functionals $J_{\vc{Q},1}$ and $J_{Q,2}$.
(Note, however, that Lemma~\ref{lem:logsob} does offer a partial 
result showing that the relative entropy can be bounded 
in terms of $J_{\vc{Q},1}$.)

To illustrate the connection between these two information 
functionals and Stein's method, we find it simpler 
to first examine the Katti-Panjer information.
Recall that,
for an arbitrary function $h:\Zpl\to\RL$, 
a function $g:\Zpl\to\RL$ satisfies 
the {\em Stein equation} for the compound Poisson
measure $\Cpo$ if,
\begin{equation}
\lambda\sum_{j=1}^\infty jQ(j)g(k+j)=kg(k)+\Big[h(k)-E[h(Z)]\Big],
\;\;\;\;g(0)=0,
\label{eq:stein}
\end{equation}
where $Z\sim\Cpo$.
[See, e.g., \cite{Erh05} for details
as well as a general review of Stein's method 
for Poisson and compound Poisson approximation.]
Letting $h=\IND_A$ for some $A\subset\Zpl$,
writing $g_A$ for the corresponding solution
of the Stein equation, 
and taking expectations with respect to an
arbitrary random variable $Y\sim P$ on $\Zpl$,
$$P(A)-\Pr\{Z\in A\}=
E\left\{ \lambda\sum_{j=1}^\infty jQ(j)g_A(Y+j)-Yg_A(Y)\right\}.$$
Then taking absolute values and maximizing
over all $A\subset\Zpl$, 
\begin{equation}
\label{eq:steinbasic}
d_{\rm TV}(P,\Cpo)
\leq
	\sup_{A\subset\Zpl}
	\left|
	E\left\{ \lambda\sum_{j=1}^\infty jQ(j)g_A(Y+j)-Yg_A(Y)\right\}
	\right|.
\end{equation}
Noting that the expression in the expectation
above is reminiscent of the Katti-Panjer recursion
(\ref{eq:panjer}), it is perhaps not surprising that
this bound relates directly to the 
Katti-Panjer information functional:

\begin{prop}
\label{prop:jloc}
For any random variable $Y\sim P$ on $\Zpl$,
any distribution $Q$
on $\NN$ and any $\lambda>0$,
$$
d_{\rm TV}(P,\Cpo)
\leq
H(\lambda,Q)\sqrt{J_{Q,2}(Y)},
$$
where $H(\lambda,Q)$ is the Stein factor defined
in {\em (\ref{eq:steinF})}.
\end{prop}
\begin{proof}
We assume without loss of generality that $Y$
is supported on the whole of $\Zpl$, since, otherwise,
$J_{Q,2}(Y)=\infty$ and the result is trivial.
Continuing from the inequality in (\ref{eq:steinbasic}),
\ben
d_{\rm TV}(P,\Cpo)
&\leq&
	\Big(\sup_{A\subset\Zpl}\|g_A\|_\infty\Big)
		\sum_{y=0}^\infty
		\left|
		 \lambda\sum_{j=1}^\infty jQ(j)P(y-j)-yP(y)
		\right|\\
&\leq&
	H(\lambda,Q)
		\sum_{y=0}^\infty
		P(y) |r_2(y;P,Q)|\\
&\leq&
	H(\lambda,Q)\sqrt{
		\sum_{y=0}^\infty
		P(y) r_2(y;P,Q)^2},
\een
where the first inequality follows from rearranging
the first sum, the second inequality follows from 
Lemma~\ref{lem:stein2} below, and the last step
is simply the Cauchy-Schwarz inequality.
\end{proof}

The following uniform bound on the sup-norm
of the solution to the Stein equation (\ref{eq:stein})
is the only auxiliary result we require from
Stein's method. See \cite{BCL92} or \cite{Erh05} 
for a proof.

\begin{lem}\label{lem:stein2}
If $g_A$ is the solution to the Stein equation
{\em (\ref{eq:stein})} for $g=\IND_A$, with $A\subset\Zpl$,
then $\|g_A\|_{\infty} \leq H(\lam, Q)$, 
where $H(\lambda,Q)$ is the Stein factor defined
in {\em (\ref{eq:steinF})}.
\end{lem}

\subsection{Size-biased information dominates total variation}
\label{ss:tv-SB}

Next we establish an analogous bound to that of 
Proposition~\ref{prop:jloc} for the size-biased
information $J_{\vc{Q},1}$. As this
functional is not as directly related to the 
Katti-Panjer recursion (\ref{eq:panjer})
and the Stein equation (\ref{eq:steinbasic}),
the proof is technically more involved.

\begin{prop}\label{prop:sb-stein}
Consider a sum $S = \Sumn Y_i\sim P$ 
of independent random variables $Y_i=B_iX_i$,
where each $X_i$ has distribution $Q_i$ on $\NN$
with mean $q_i$, and each $B_i\sim\Bern(p_i)$.
Let $\lam = \sum_{i=1}^n p_i$ and
$Q=\sum_{i=1}^n\frac{p_i}{\lambda}Q_i$.
Then,
\ben
d_\stv(P , \Cpo) \leq 
H(\lambda,Q) q \left( \sqrt{ \lambda J_{\vc{Q},1}(S)} + D(\vc{Q})
\right),
\een
where $H(\lam,Q)$ is the Stein factor
defined defined in {\em (\ref{eq:steinF})},
$q=(1/\lambda)\sum_ip_iq_i$ is the mean of $Q$,
and $D(\vc{Q})$ is the
measure of the dissimilarity between the 
distributions $\vc{Q}=(Q_i)$,
defined in~{\em (\ref{eq:DQ})}.
\end{prop}
\begin{proof} 
For each $i$, let $T^{(i)}\sim F^{(i)}$ denote
the leave-one-out sum $\sum_{j\neq i}Y_i$,
and note that, as in the proof of 
Corollary~\ref{cor:j1simple}, the distribution
$F^{(i)}$ is the same as the distribution $P^{(i)}$
of the modified sum $S^{(i)}$ in 
Definition~\ref{defn:SBLI-basic}.
Since $Y_i$ is nonzero with probability $p_i$,
we have, for each $i$,
\ben
E[Y_{i} g_{A}(S) ]
&=&
	E[Y_ig_A(Y_i+T^{(i)})]\\
&=&
	\sum_{j=1}^\infty\sum_{s=0}^\infty p_i Q_i(j) F^{(i)}(s) jg_A(j+s)\\
&=&
	\sum_{j=1}^\infty\sum_{s=0}^\infty p_i j Q_i(j) P^{(i)}(s) g_A(s+j),
\een
where, for $A\subset\Zpl$ arbitrary, 
$g_A$ denotes the solution of the Stein
equation (\ref{eq:stein}) with $h=\IND_A$.
Hence,
summing over $i$ and substituting 
in the expression in the
right-hand-side of
equation~(\ref{eq:steinbasic}) 
with $S$ in place of $Y$,
yields,
\be
	E\Big\{ \lambda\sum_{j=1}^\infty 
&&
	\hspace*{-0.35in}
	jQ(j)g_A(S+j)
	\Big\}
	-E[Sg_A(S)]
	\nonumber\\
& = &  
	\Suminfs \suminfj g_A(s+j) \left(  \lam jQ(j)P(s)
	- \sum_i p_i j Q_i(j) P^{(i)}(s) \right)  
	\nonumber\\ 
& = &
	\Suminfs \suminfj g_A(s+j) j Q(j) \left(  
	\sum_i p_i (P(s) - P^{(i)}(s)) \right)  
	\nonumber\\ 
& & 
	+  \Suminfs \suminfj g_A(s+j) 
	\left(  \sum_i p_i j (Q(j) - Q_i(j)) P^{(i)}(s) \right) 
	\nonumber \\ 
& = & 
	- \Suminfs \suminfj g_A(s+j) j Q(j) \lambda P(s) \left(  
	\frac{\sum_i p_i P^{(i)}(s)}{\lambda P(s)} - 1 \right)  
	\nonumber\\ 
& & 
	+  \Suminfs \suminfj g_A(s+j) 
	\left(  \sum_i p_i j (Q(j) - Q_i(j)) P^{(i)}(s) \right) .
	\label{eq:problem}
\ee
By the Cauchy-Schwarz inequality, 
the first term in (\ref{eq:problem}) is
bounded in absolute value by,
\ben
 \sqrt{ \lam \sum_{j,s} g_A(s+j)^2 j Q(j) P(s)}
\sqrt{ \lam \sum_{j,s} jQ(j) P(s)
 \left(  \frac{\sum_i p_i P^{(i)}(s)}{\lam P(s)} - 1 \right)^2 },
\label{eq:CS1}
\een
and for the second term, simply bound $\| g_A \|_{\infty}$
by $H(\lambda,Q)$ using Lemma~\ref{lem:stein2},
deducing a bound in absolute value of 
$$ H(\lambda,Q) \sum_{i,j} p_i j | Q(j) - Q_i(j)|.$$

Combining these two bounds with the expression 
in (\ref{eq:problem}) and the original total-variation
inequality (\ref{eq:steinbasic}) completes the proof,
upon substituting the uniform sup-norm bound 
given in Lemma~\ref{lem:stein2}.
\end{proof}

Finally, recall from the discussion in the beginning
of this section that the scaled Fisher information 
$J_{\pi}$ satisfies a modified log-Sobolev inequality 
(\ref{eq:logS}), which gives a bound for the relative 
entropy in terms of the functional $J_\pi$.
For the information functionals 
$J_{\vc{Q},1}$ and $J_{Q,2}$ considered
in this work, we instead established
analogous bounds in terms of total variation.
However, the following partial result holds
for $J_{\vc{Q},1}$:

\begin{lem} \label{lem:logsob}
Let $Y = BX\sim P$, where $B\sim\Bern(p)$
and $X \sim Q$ on $\NN$. Then:
$$ D( P \| {\rm CPo}(p,Q) ) \leq J_{Q,1}(Y).$$ 
\end{lem}
\begin{proof}
Recall from (\ref{eq:j1oneterm}) 
that $J_{Q,1}(Y) = \frac{p^2}{1-p}$.
Further, since 
the ${\rm CPo}(p,Q)$ mass function at $s$
is at least $e^{-p} p Q(s)$
for $s \geq 1$, 
we have,
$D( C_{Q}\Bern(p) \| {\rm CPo}(p,Q) ) 
\leq (1-p) \log(1-p) + p$, which yields the result.
\end{proof}


\section{Comparison with existing bounds}
\label{sec:num}
In this section, we compare the bounds obtained 
in our three main results, 
Theorems~\ref{thm:1},~\ref{thm:2} and
\ref{thm:3}, with  inequalities derived by other methods.
Throughout, $S_n=\Sumn Y_i = \Sumn B_iX_i$, where the $B_i$ and the~$Y_i$ are
independent sequences of independent random variables, with $B_i\sim\Bern(p_i)$ 
for some $p_i\in(0,1)$, and with $X_i\sim Q_i$ on $\NN$; we write $\l = \Sumn p_i$. 

There is a large body of literature developing
bounds on the distance between the distribution~$P_{S_n}$ of~$S_n$ and 
compound Poisson distributions;
see, e.g., \cite{Erh05} and the references therein,
or \cite[Section~2]{Roo03} for a concise review.

We begin with the case in which all the~$Q_i=Q$ are identical, when, in view
of a remark of Le~Cam \cite[bottom of p.187]{LeC65} 
and Michel \cite{Mic87},
bounds computed for the case 
$X_i=1$ a.s.\ for all~$i$ are also valid for any~$Q$.
One of the earliest results is the following
inequality of Le Cam \cite{LeC60}, 
building on earlier results by Khintchine and Doeblin,
\be
d_\stv(P_{S_n},\Cpo) \leq \Sumn p_i^2.
\label{lecam-bound}
\ee
Barbour and Hall~(1984) used Stein's method to improve the bound to
\be
d_\stv(P_{S_n},\Cpo) \leq \min\{1,\l^{-1}\}\Sumn p_i^2.
\label{BH-bound}
\ee
Roos \cite{Roo01} gives the asymptotically sharper bound
\be
\label{Roos-Q-1}
d_\stv(P_{S_n},\Cpo) \leq 
    \Bl \frac3{4e} + \frac{7\sth(3-2\sth)}{6(1-\sth)^2}\Br\,\th,
\ee
where $\th = \l^{-1}\Sumn p_i^2$.  
In this setting, the bound~(\ref{eq:ident2})
that was derived from Theorem~\ref{thm:1} yields
\begin{equation}
  d_\stv(P_{S_n} , \Cpo) 
     \leq \Bl \frac{1}{2 \l} \Sumn \frac{p_i^3}{1-p_i} \Br^{1/2}.
\label{eq:j1simple}
\end{equation}
The bounds (\ref{BH-bound}) -- (\ref{eq:j1simple})
are all derived using the observation
made by Le Cam and Michel, taking~$Q$ to be degenerate at~$1$.  
For the application of Theorem~\ref{thm:3}, however, 
the distribution~$Q$ must have support the whole of~$\NN$, 
so~$Q$ cannot be replaced by the point mass at~$1$ 
in the formula; the bound that results 
from Theorem~\ref{thm:3} can be expressed as
\begin{equation}
d_\stv( P_{S_n} , \Cpo )
   \leq H(\lam,Q) \Bl K(Q) \Sumn p_i^3 \Br^{1/2},\quad
     \mbox{with}\ K(Q) = \sum_y Q(y) y^2 \left( 
	\frac{ Q^{*2}(y)}{2 Q(y)} - 1 \right)^2 .
 \label{eq:j3more}
\end{equation}
Illustration of the effectiveness of these bounds with geometric~$Q$ and equal~$p_i$
is given in Section~\ref{sec:equal-Q}.

For non-equal~$Q_i$, the bounds are more complicated.  
We compare those given in
Theorems \ref{thm:2} and~\ref{thm:3} with three other bounds. 
The first is Le Cam's bound (\ref{lecam-bound}) that
still remains valid as stated in the case of non-equal $Q_i$.
The second, from
Stein's method, has the form
\be
\label{BCL-when-decr}
d_\stv(P_{S_n},\Cpo) \leq G(\lam,Q) \Sumn q_i^2 p_i^2,
\ee
see Barbour and Chryssaphinou \cite[eq.~(2.24)]{BCh01},
where $q_i$ is the mean of~$Q_i$ and $G(\lam,Q)$ is a Stein factor: 
if $jQ(j)$ is non-increasing, then
\ben
G(\lam,Q) = \min \bigg\{ 1, \,\, \delta \bigg[ \frac{\delta}{4} + 
\log^{+}\bigg( \frac{2}{\delta} \bigg) \bigg] \bigg\} ,
\een
where $\delta=[\lam\{Q(1)-2Q(2)\}]^{-1}\ge0$. 
The third is that of
Roos \cite{Roo03}, Theorem~2, which is in detail very complicated, but
correspondingly accurate.  A simplified version, 
valid if $jQ(j)$ is decreasing,
gives
\be
\label{Roos-Q-2}
d_\stv(P_{S_n},\Cpo) \leq 
     \frac{\a_2}{(1-2e\a_2)_+}, 
\ee
where
\[
   \a_2 = \Sumn g(2p_i)p_i^2\min\Bl \frac{q_i^2}{e\l},\frac{\n_i}{2^{3/2}\l},1 \Br,
\]
$\n_i = \sum_{y\ge1}Q_i(y)^2/Q(y)$ and $g(z) = 2z^{-2}e^z(e^{-z}-1+z)$.
We illustrate the effectiveness of these bounds in Section~\ref{sec:different-Q};
in our examples, Roos's bounds are much the best.

\subsection{Broad comparisons}
\label{sec:broad}

Because of their apparent complexity and different
forms, general comparisons between the bounds 
are not straightforward,
so we consider two particular cases below in 
Sections \ref{sec:equal-Q} and~\ref{sec:different-Q}. 
However, the following simple observation on approximating 
compound binomials by a compound Poisson gives a first 
indication of the strength of one of our bounds.

\newpage

\begin{prop}
For equal $p_i$ and equal~$Q_i$:
\begin{enumerate}
\item
If $n > (\sqrt{2}p(1-p))^{-1}$,
then the bound of Theorem~{\em \ref{thm:1}}
is stronger than Le Cam's bound {\em (\ref{lecam-bound})};
\item
If $p < 1/2$, 
then the bound of Theorem~{\em \ref{thm:1}}
is stronger than the bound {\em (\ref{BH-bound})};
\item 
If  $0.012 < p < 1/2$ and
$n>(\sqrt{2}p(1-p))^{-1}$ are satisfied,
then the bound of Theorem~{\em \ref{thm:1}}
is stronger than all three bounds
in {\em (\ref{lecam-bound})}, {\em (\ref{BH-bound})}
and {\em (\ref{Roos-Q-1})}.
\end{enumerate}
\end{prop}

\begin{proof}
The first two observations follow by simple algebra,
upon noting that the bound of Theorem~\ref{thm:1} 
in this case reduces to $\frac{p}{\sqrt{2(1-p)}}$;
the third is shown numerically, noting that here $\th=p$.
\end{proof}

\medskip
One can also examine the rate of convergence of 
the total variation distance between the 
distribution $P_{S_n}$ and the 
corresponding compound Poisson distribution, under simple
asymptotic schemes.  We think of situations in which the
$p_i$ and~$Q_i$ are not necessarily equal, but are all in
some reasonable sense comparable with one another; we shall
also suppose that $jQ(j)$ is more or less a fixed and decreasing
sequence.  Two ways in which~$p$ varies with~$n$ are considered:
\begin{enumerate}
\item[]{\bf Regime I. } $p=\lambda/n$ for fixed $\lambda$,
and $n\to\infty$;
\item[]{\bf Regime II. } 
$p=\sqrt{\frac{\mu}{n}}$, so that 
$\lam=\sqrt{\mu n}\ra\infty$ as $n\to\infty$.
\end{enumerate}

Under these conditions, the Stein factors
$H(\lambda,Q)$ are of the same order as $1/\sqrt{np}$.
Table~\ref{tab:compare-cpa} compares the
asymptotic performance of the various bounds above. 

\begin{table}[ht]
\begin{center}
\begin{tabular}{|l|c|c|c|}
\hline
{\bf Bound} &  $d_{\rm TV}(P_{S_n} , {\rm CPo}(\lam,Q))$ to leading order
& {\bf I} & {\bf II} \\
\hline\hline
Le Cam (\ref{lecam-bound}) & $np^2$ & $n^{-1}$ & 1 \\
\hline
Roos (\ref{Roos-Q-2}) & $ np^2 \min(1,1/(np))$ &
$n^{-1}$ & $ n^{-1/2}$ \\ 
\hline
Stein's method (\ref{BCL-when-decr}) & 
$n p^2 \min(1,\log(np)/np)$ & $n^{-1}$ & $ n^{-1/2} \log n$ \\
\hline
Theorem~\ref{thm:2} & 
$p$ & $1$ & $ n^{1/4}$  \\
\hline
Theorem~\ref{thm:3}
(\ref{eq:j3more}) & 
$p$ & $n^{-1}$ & $ n^{-1/2}$  \\
\hline
\end{tabular}
\end{center}
\caption{Comparison of the first-order asymptotic performance 
of the bounds in 
(\ref{lecam-bound}), (\ref{BCL-when-decr}) and~(\ref{Roos-Q-2}),
with those of Theorems \ref{thm:2} and~\ref{thm:3} 
for comparable but non-equal~$Q_i$,
in the two limiting regimes $p\asymp 1/n$ and $p \asymp 1/\sqrt{n}$.}
\label{tab:compare-cpa}
\end{table}

The poor behaviour of the bound in Theorem~\ref{thm:2}
shown above occurs because, for large values of $\lambda$,
the quantity $D(\vc{Q})$ behaves much 
like $\lambda$, unless the~$Q_i$ are identical 
or near-identical.

\subsection{Example. Compound binomial with equal geometrics}
\label{sec:equal-Q}

We now examine the finite-$n$ behavior of 
the approximation bounds (\ref{lecam-bound}) -- (\ref{Roos-Q-1})
in the particular case of equal~$p_i$ and equal~$Q_i$,
when $Q_i$ is geometric with parameter
$\alpha>0$, $Q(j)=(1-\alpha)\alpha^{j-1}$,
$j\geq 1$. 

If $\alpha<\half$, then 
$\{j Q(j)\}$ is decreasing and, with
$\delta=[\lam (1-3\alpha+2\alpha^2) ]^{-1}$,
the Stein factor in (\ref{eq:j3more}) becomes
\ben 
H(\lam, Q)= \min\{1,\sqrt{\delta}(2-\sqrt{\delta})\}. 
\een
The resulting bounds are plotted 
in Figures~\ref{fig:2a}~--~\ref{fig:2c}. 

\begin{figure}[!htbp]
\centering
\includegraphics*[angle=0,width=3.5in]{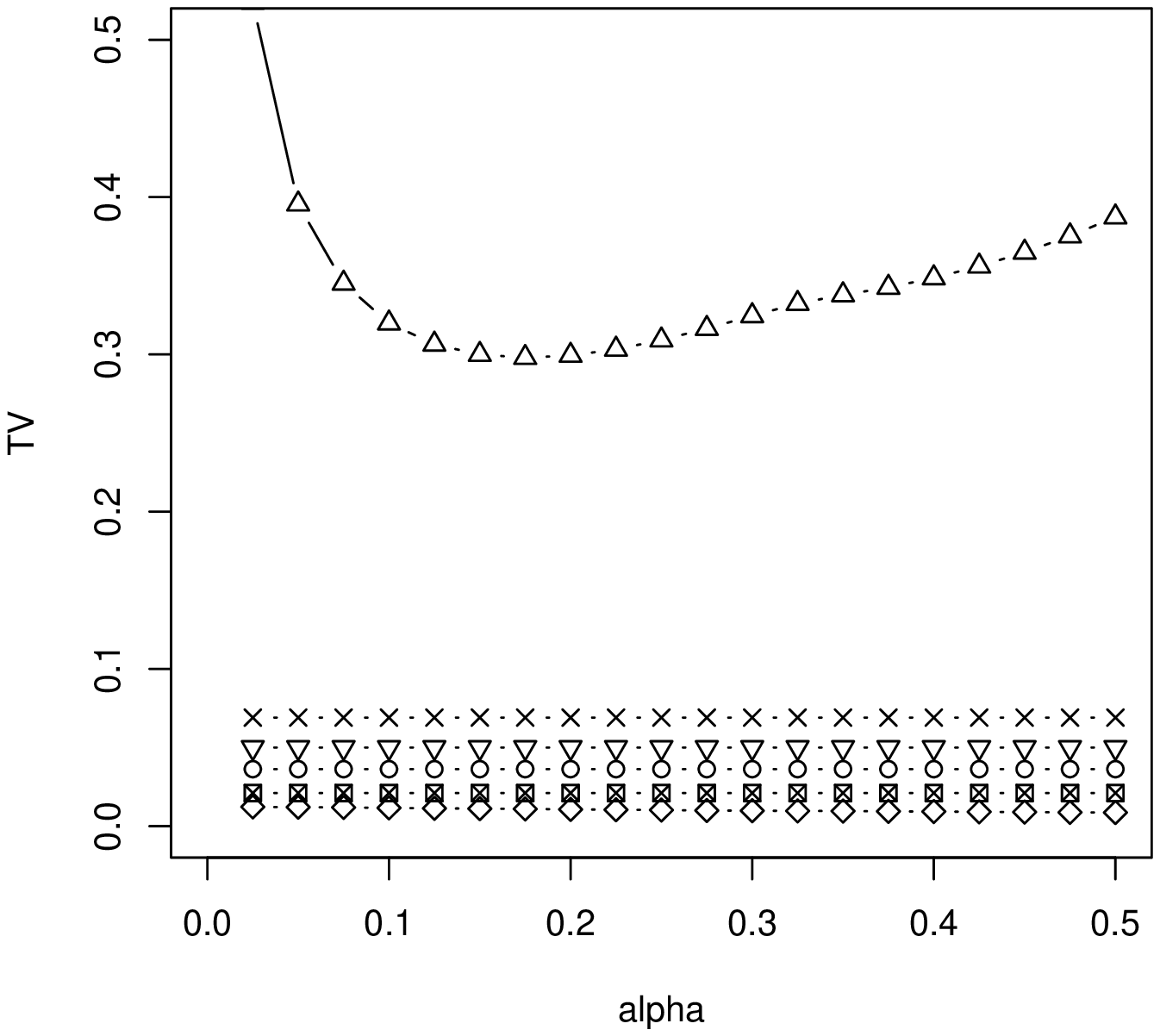}
\caption{
\label{fig:2a} 
Bounds on the total variation distance 
$d_{\rm TV}(C_Q{\rm Bin}(p,Q),\Cpo)$ for $Q\sim{\rm Geom}(\alpha)$, 
plotted against
the parameter $\alpha$, with $n = 100$ and $\lam = 5$ fixed.
The values of the bound in~(\ref{eq:j1simple})
are plotted as $\circ$; those in~(\ref{eq:j3more}) as
$\triangle$; 
those of the Stein's method bound in~(\ref{BH-bound})
as $\bigtriangledown$;
and Roos' bounds in~(\ref{Roos-Q-1}) 
and~(\ref{Roos-Q-2}) as~$\times$ and $\boxtimes$, respectively.
The true total variation distances, computed numerically in each
case, are plotted as $\diamond$.}
\end{figure}

\begin{figure}[!htbp]
\centering
\includegraphics*[angle=0,width=3.5in]{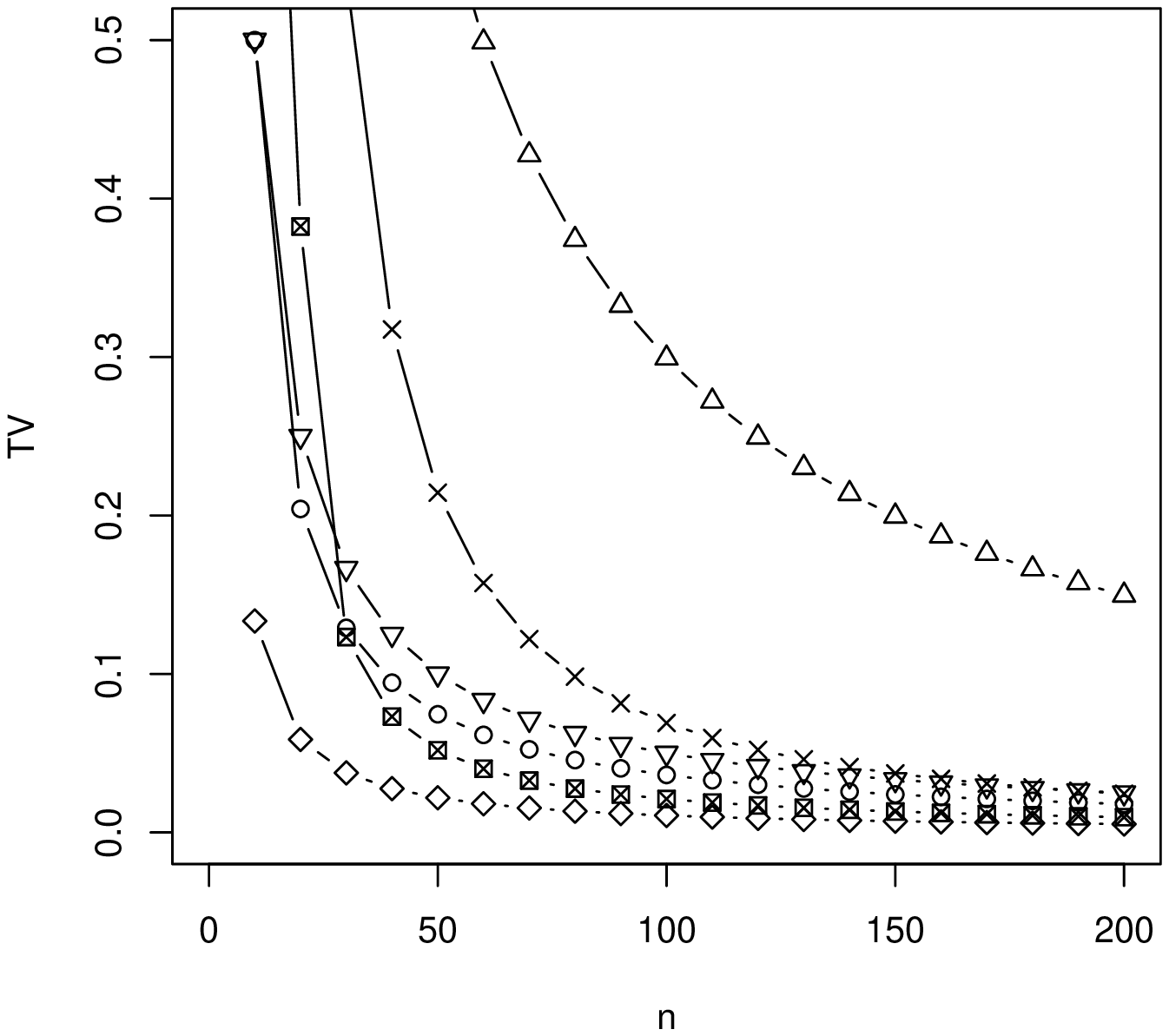}
\caption{
\label{fig:2b} 
Bounds on the total variation distance 
$d_{\rm TV}(C_Q{\rm Bin}(p,Q),\Cpo)$ for $Q\sim{\rm Geom}(\alpha)$
as in Figure~\ref{fig:2a}, here plotted against 
the parameter $\alpha$, with $n = 100$ and $\lam = 5$ fixed.}
\end{figure}

\begin{figure}[!htbp]
\centering
\includegraphics*[angle=0,width=3.5in]{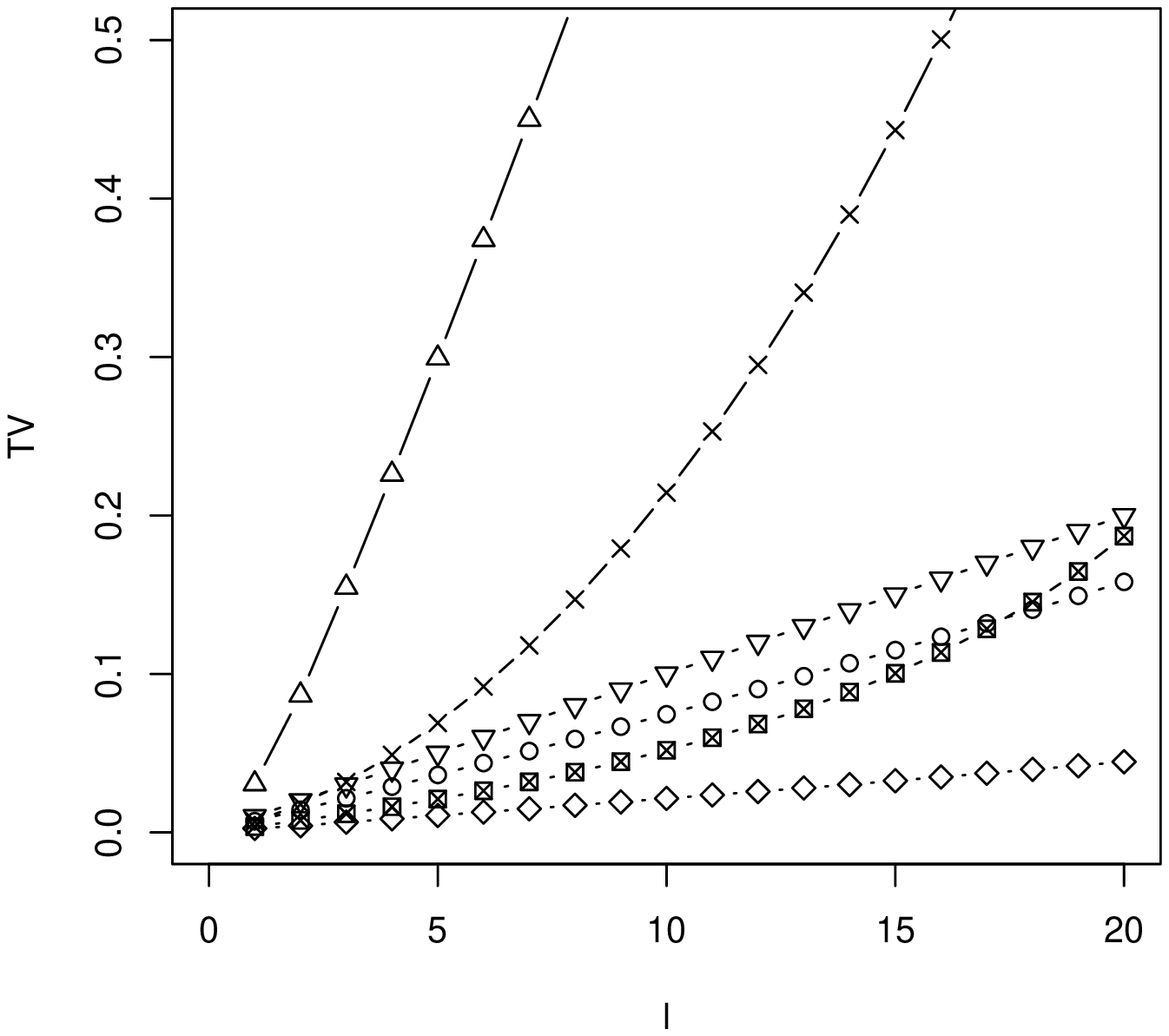}
\caption{
\label{fig:2c} 
Bounds on the total variation distance 
$d_{\rm TV}(C_Q{\rm Bin}(p,Q),\Cpo)$ for $Q\sim{\rm Geom}(\alpha)$
as in Figure~\ref{fig:2a}, here plotted against
the parameter $\lam$, with $\alpha=0.2$ and $n=100$ fixed.}
\end{figure}

\subsection{Example. Sums with unequal geometrics}
\label{sec:different-Q}

Here, we consider finite-$n$ behavior of 
the approximation bounds (\ref{lecam-bound}), (\ref{BCL-when-decr}) and
(\ref{Roos-Q-2}) in the particular
case when the distributions $Q_i$ are geometric with parameters
$\alpha_i>0$.
The resulting bounds are plotted 
in Figures~\ref{fig:3a} and~\ref{fig:3b}.

In this case, it is clear that the best bounds by a considerable margin are
those of Roos~\cite{Roo03} given in~(\ref{Roos-Q-2}).

\begin{figure}[!htbp]
\centering
\includegraphics[angle=0,width=3.5in]{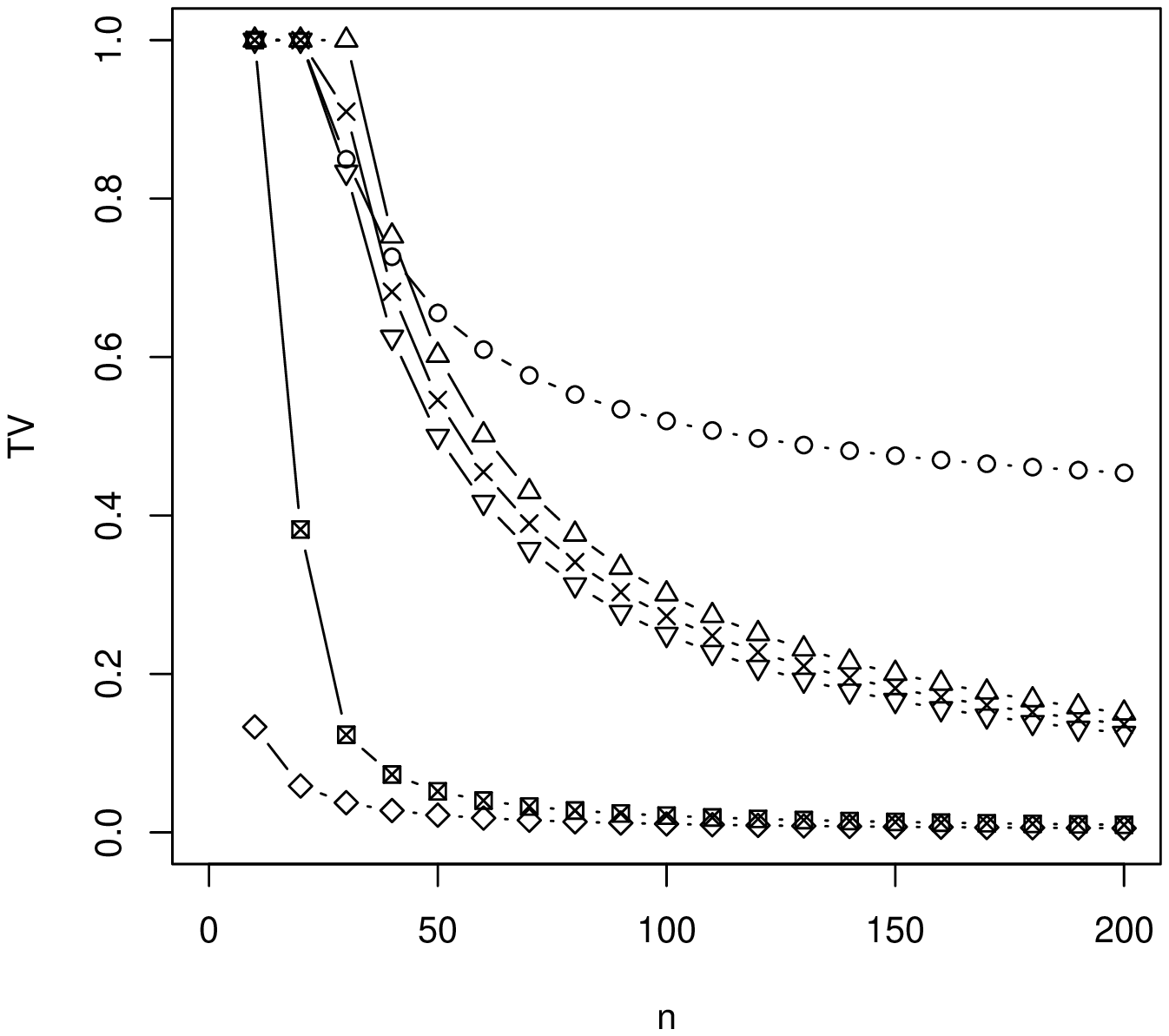}
\caption{
\label{fig:3a} 
Bounds on the total variation distance 
$d_{\rm TV}(P_{S_n},\Cpo)$ for $Q_i \sim{\rm Geom}(\alpha_i)$,
where $\alpha_i$ are uniformly spread between $0.15$ and $0.25$, 
$n$ varies, and $p$ is as in regime I, $p=5/n$.
Again, bounds based on $J_{\vc{Q},1}$ 
are plotted as $\circ$; those based on $J_{Q,2}$ as
$\triangle$; Le Cam's bound in~(\ref{lecam-bound}) as~$\bigtriangledown$;
the Stein's method bound in~(\ref{BCL-when-decr}) as~$\times$,
and Roos' bound from Theorem 2 of \cite{Roo03} as~$\boxtimes$. 
The true total variation distances, computed numerically in each
case, are plotted as $\diamond$.
}
\end{figure}

\begin{figure}[!htbp]
\centering
\includegraphics[angle=0,width=3.5in]{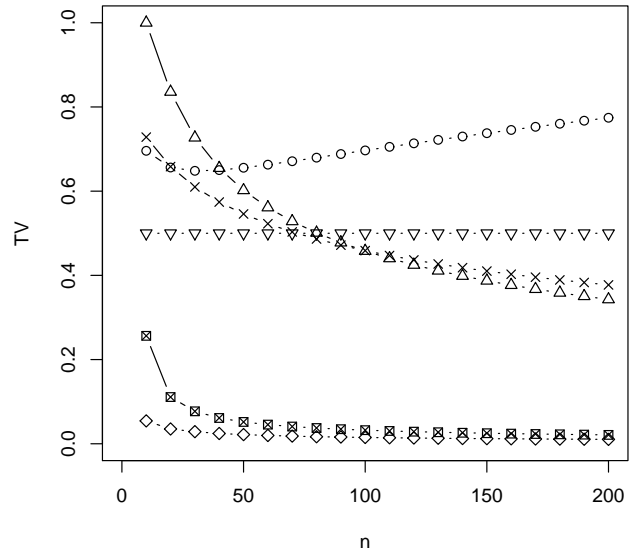}
\caption{
\label{fig:3b} 
Bounds on the total variation distance 
$d_{\rm TV}(P_{S_n},\Cpo)$ for $Q_i \sim{\rm Geom}(\alpha_i)$
as in Figure~\ref{fig:3a},
where $\alpha_i$ are uniformly spread between $0.15$ and $0.25$, 
$n$ varies, and $p$ is as in Regime II, $p=\sqrt{0.5/n}$.
}
\end{figure}



\end{document}